\documentclass[11pt,reqno]{amsart}
\usepackage{amssymb,amsmath,amsthm,epsfig,times}
\usepackage{mathrsfs}
\usepackage{epstopdf}
\usepackage{pdfpages}
\numberwithin{equation}{section} %sets equation numbers <chapter>.<section>.<index>
\usepackage{verbatim}
\usepackage[pdftex]{hyperref}
\def\C{{\mathbb C}}

\def\P{{\mathbb P}}

\def\*S{{\mathbb S}}

\newcommand{\dis}{\displaystyle}
\newcommand{\cal}{\mathcal}

\newtheorem{theorem}{Theorem}
\newtheorem{proposition}{Proposition}[section]
\newtheorem{lemma}[proposition]{Lemma}

\begin{document}
\title[Residual Cauchy-type formula]
{Residual Cauchy-type formula on Riemann surfaces.}

\author{Peter L. Polyakov}
\address{University of Wyoming \\ Department of Mathematics \\ 1000 E University Ave
\\ Laramie, WY 82071}
\curraddr{}
\email{polyakov@uwyo.edu}

\subjclass[2010]{Primary: 30F, 32A10, 32A26}

\keywords{Riemann surfaces, Cauchy formula}

\dedicatory{In memory of Gennadi Henkin}

\begin{abstract}
We construct a Cauchy type formula on open subdomains of
Riemann surfaces of the form $V=\left\{z\in \C\P^2: P(z)=0\right\}$.
\end{abstract}

\maketitle

%\centerline{\DATE}

\section{Introduction.}\label{Introduction}

\indent
Let $V\subset {\widetilde V}$ be a smoothly bordered open subset of a Riemann surface
\begin{equation}\label{ProjectiveCurve}
{\widetilde V}=\left\{z\in \C\P^2:\ P(z)=0\right\}.
\end{equation}
The goal of the present article is the construction of a Cauchy-type integral formula
defining the values of a holomorphic function $f$ on $V$ through its values on the boundary $bV$.
The article represents a continuation of the line of research developed in the joint articles
with Gennadi Henkin \cite{HP2} and \cite{HP3} and devoted to application of integral
formulas to complex analysis on subvarieties of projective spaces.\\
\indent
Throughout the article we assume that a holomorphic function $f$, for which we construct
a boundary integral formula, is defined on some tubular neighborhood
of $V$ in $\C\P^2$. As in \cite{HP1}, we identify $f$ on a domain in $\C\P^2$ with
its lift to a domain in $\*S^{5}(1)$ satisfying appropriate homogeneity conditions.
Then the sought formula is constructed as the residue of the formula on a tubular domain
\begin{equation}\label{Uepsilon}
U^{\epsilon}=\left\{z\in \*S^{5}(1):\ \left|P(z)\right|<\epsilon, \varrho(z)<0\right\},
\end{equation}
where we assume that $V=\left\{z\in {\widetilde V}: \varrho(z)<0\right\}$.\\
\indent
Integral formula constructed in the article uses two types of barrier-functions. The first
barrier, which is described in the lemma below is a Weil-type barrier constructed
by A. Weil in \cite{W} (see also \cite{HP2}).

%******************************************************************
%******* Lemma WeilCoefficients  **************************************
\begin{lemma}\label{WeilCoefficients} Let $P(\zeta)$ be a homogeneous polynomial of variables
$\zeta_0,\zeta_1,\zeta_2$ of degree $d$. Then there exist polynomials
$\left\{Q^i(\zeta,z)\right\}_{i=0}^2$ satisfying:
\begin{equation}\label{HomogeneityConditions}
\left\{\begin{array}{ll}
P(\zeta)-P(z)=\sum_{i=0}^2Q^i(\zeta,z)\cdot\left(\zeta_i-z_i\right),\vspace{0.1in}\\
Q^i(\lambda\zeta,\lambda z)=\lambda^{d-1}\cdot Q^i(\zeta,z)\ \text{for}\ \lambda\in\C.
\end{array}\right.
\end{equation}
\end{lemma}
%******************************************************************
%******* End Lemma WeilCoefficients  ***********************************

\indent
The choice of the second barrier is adjusted to the boundary $bV$, but also
depends on the global structure of the Riemann surface ${\widetilde V}$.
The second barrier has the form
\begin{equation}\label{BoundaryBarrier}
F(z,\zeta)=\sum_{i=0}^2R_i(z)(\zeta_i-z_i)
\end{equation}
with coefficients $R_i$ holomorphically depending on $z$. From this point on we assume that
for an arbitrary $z\in V$ there exists a vector field $R(w)=\left[R_0(w), R_1(w), R_2(w)\right]$
on $V$ with values in the conormal vector bundle of $V$, satisfying the following conditions
\begin{equation}\label{RConditions}
\begin{array}{ll}
(i)\quad R(w)\neq 0\ \text{for any}\ w\in V,\vspace{0.1in}\\
(ii)\quad \text{there exists a neighborhood}\ {\cal V}_z\ni z\ \text{such that for}\
w\in {\cal V}_z\ \text{the set}\ {\cal S}(w)=\left\{\zeta\in V: F(w,\zeta)=0\right\}\\
\text{consists of finitely many points}\ \left\{w^{(0)}=w, w^{(1)},\dots, w^{(p)}\right\},
\text{at which the line} \left\{\zeta: F(w,\zeta)=0\right\}\\
\text{transversally intersects}\ V.
\end{array}
\end{equation}
\indent
Existence of a vector function $R(w)$ satisfying condition (i) of \eqref{RConditions}
follows from the triviality of the conormal vector bundle over the open Riemann surface $V$
(see \cite{Fo}), and condition (ii) can be achieved without loss of generality with the use
of the Bertini's Theorem (see \cite{Ha}) by slightly changing the field $R$.

\indent
Using barriers \eqref{HomogeneityConditions} and \eqref{BoundaryBarrier} we can now
formulate the main theorem of the article.

%******************************************************************
%******* Theorem Main  *********************************************
\begin{theorem}\label{Main} Let $V\subset \left\{\C\P^2\setminus\left\{z_0=0\right\}\right\}$
and $U^{\epsilon}$ be as in \eqref{ProjectiveCurve} and \eqref{Uepsilon} respectively,
and let $f$ be a holomorphic function of negative homogeneity in $U^{\epsilon}$
for some $\epsilon>0$. Let $z\in V$ be a fixed point, and let ${\cal V}_z\ni z$
be a neighborhood of $z$ in $V$, such that
\begin{itemize}
\item[(i)] conditions \eqref{RConditions} are satisfied,
\item[(ii)] function ${\dis \frac{w_1}{w_0} }$ takes distinct values at the points
$\left\{w,w^{(1)},\dots,w^{(p)}\right\}$ of ${\cal S}(w)$ for $w\in {\cal V}_z$.
\end{itemize}
\indent
Then for $w\in {\cal V}_z$ the following equality holds for the values of $f$ at the points
of ${\cal S}(w)$:
\begin{equation}\label{fValues}
f(w^{(k)})=\frac{p+1}{p+2}\cdot\frac{\det A_k(w,w^{(1)},\dots,w^{(p)})}{\det A(w)},
\end{equation}
where $A(w)$ is the Vandermonde matrix
\begin{equation}\label{Vandermonde}
A(w)=\left[\begin{tabular}{cccc}
1&1&$\cdots$&1\vspace{0.05in}\\
${\dis \frac{w_1}{w_0} }$&${\dis \frac{w^{(1)}_1}{w_0} }$&$\cdots$&
${\dis \frac{w^{(p)}_1}{w^{(p)}_0} }$\vspace{0.05in}\\
$\vdots$&$\vdots$&&$\vdots$\vspace{0.05in}\\
${\dis \left(\frac{w_1}{w_0}\right)^p }$&
${\dis \left(\frac{w^{(1)}_1}{w_0}\right)^p }$&$\cdots$&
${\dis \left(\frac{w^{(p)}_1}{w^{(p)}_0}\right)^p }$
\end{tabular}\right],
\end{equation}
$A_k(w,w^{(1)},\dots,w^{(p)})$ is the matrix $A(w)$ with the $k$-th column
replaced by the column
\begin{multline}\label{GValues}
G_k(w,w^{(1)},\dots,w^{(p)})
=\frac{2}{(2\pi i)^{3}}\Bigg(\sum_{j=0}^p
\lim_{\epsilon\to 0}\int_{\Gamma^{\epsilon}}
f(\zeta)\cdot\left(\frac{\zeta_1}{\zeta_0}\right)^k\\
\times\det\left[\frac{Q(\zeta,w^{(j)})}{P(\zeta)}\
\frac{R(w^{(j)})}{F(w^{(j)},\zeta)}\ \frac{\bar\zeta}{B(\zeta,w^{(j)})}\right]
d\zeta_0\wedge d\zeta_1\wedge d\zeta_2\Bigg),
\end{multline}
and
\begin{equation}\label{Gamma}
\Gamma^{\epsilon}=\left\{z\in \*S^{5}(1):\ \left|P(z)\right|=\epsilon, \varrho(z)=0\right\}.
\end{equation}
\end{theorem}
%******************************************************************
%******* End Theorem Main  ******************************************

{\bf Remark 1.} Assumption of negativity of the homogeneity of $f$ in Theorem~\ref{Main}
is made just for shortening the resulting formula. If $f$ has homogeneity $\ell\geq 0$, then
Theorem~\ref{Main} can be applied to function $f(z)\cdot z_0^{-\ell-1}$ and the result
multiplied by $z_0^{\ell+1}$.\\

{\bf Remark 2.} From the formulation of Theorem~\ref{Main} it can be seen that the
Cauchy-type barrier from \eqref{BoundaryBarrier} is local with respect to $z$ and global
with respect to $\zeta$, which represents an important new ingredient in formula
\eqref{fValues}.

\section{Integral formulas on domains in $\C\P^2$.}
\label{Formulas}

\indent
In this section we construct a Cauchy-Weil-Leray type integral formula for functions
on a domain $U^{\epsilon}$ in $\C\P^2$. We start with the Koppelman-type formula
from \cite{HP1} (Proposition 1.2), which originally appeared for strictly pseudonvex manifolds
in a very important article by Henkin \cite{He2} (Theorem 3.2), and then interpreted
in \cite{HP1} for subdomains of such manifolds.
Integral formulas of this type have a long history going back to Moisil \cite{Mo},
Fueter \cite{Fe}, Martinelli \cite{Ma}, Bochner \cite{Bo}, Koppelman \cite{Kp}.\\
\indent
The proposition below is a reformulation of Proposition 1.2 from \cite{HP1}.

%************************************************************
%******* Proposition Bochner  ***********************************
\begin{proposition}\label{Bochner} Let $P$ be a homogeneous polynomial
defining the curve $V$ as in \eqref{ProjectiveCurve},
and let $f$ be a holomorphic function of homogeneity $\ell$ on the domain $U^{\epsilon}$.\\
\indent
Then the following equality is satisfied for $z\in U^{\epsilon}$
\begin{equation}\label{BochnerFormula}
f(z)=L^{\epsilon}\left[f\right](z),
\end{equation}
with
\begin{multline*}
L^{\epsilon}\left[f\right](z)
=\frac{2}{(2\pi i)^{3}}\Bigg[\int_{bU^{\epsilon}\times[0,1]}
f(\zeta)\cdot\omega^{\prime}_0\left((1-\lambda)\frac{\bar z}
{B^*(\zeta,z)}+\lambda\frac{\bar\zeta}{B(\zeta,z)}\right)\wedge d\zeta\\
+\int_{U^{\epsilon}}f(\zeta)\cdot
\omega^{\prime}_0\left(\frac{\bar\zeta}{B(\zeta,z)}\right)\wedge d\zeta\Bigg],
\end{multline*}
where
$$B^*(\zeta,z)=\sum_{j=0}^2{\bar z}_j\cdot\left(\zeta_j-z_j\right),
\hspace{0.2in}B(\zeta,z)=\sum_{j=0}^2{\bar\zeta}_j\cdot\left(\zeta_j-z_j\right),$$
$$d\zeta=d\zeta_0\wedge d\zeta_1\wedge d\zeta_2,\hspace{0.2in}
\omega^{\prime}(\eta)=\sum_{k=0}^2(-1)^{k}\eta_k\bigwedge_{j\neq k}d\eta_j,$$
and $\omega^{\prime}_0$ is the $(0,0)$-component with respect to $z$ of the form
$\omega^{\prime}$.
\end{proposition}
%*******  End Proposition Bochner  *************************************
%****************************************************************
{\bf Remark.}
The first integral in the formula above is a proper integral taken over the boundary
of $U^{\epsilon}$. The second integral is a singular integral, and, following the interpretation
in \cite{He2}, should be understood as
\begin{equation*}
\lim_{\delta\to 0}\int_{U^{\epsilon}}f(\zeta)\cdot
\omega^{\prime}_0\left(\frac{\bar\zeta}{B(\zeta,z^{\delta})}\right)\wedge d\zeta,
\end{equation*}
where $z^{\delta}=z-\delta\cdot\nu(z)$, and $\nu(z)$ is the normal to $\*S^{5}(1)$ at $z$.

\indent
We will transform the right-hand side of equality \eqref{BochnerFormula} into
a Cauchy-Weil-Leray type formula on $U^{\epsilon}$. For this transformation
we need several lemmas.

\indent
In the first lemma for convenience in the further computations we eliminate the integration
with respect to parameter $\lambda$ in the right-hand side of \eqref{BochnerFormula}
and transform the formula for operator $L^{\epsilon}$ in Proposition~\ref{Bochner}.

%***************************************************************
%******* Lemma Determinant *****************************
\begin{lemma}\label{Determinant} For operator $L^{\epsilon}$ in
Proposition~\ref{Bochner} the following equality holds:
\begin{multline}\label{LFormula}
L^{\epsilon}\left[f\right](z)
=\frac{1}{(2\pi i)^{3}}\Bigg[\int_{bU^{\epsilon}}
f(\zeta)\cdot\det\left[\frac{\bar z}{B^*(\zeta,z)}\
\frac{\bar\zeta}{B(\zeta,z)}\ \bar\partial\left(\frac{\bar\zeta}{B(\zeta,z)}\right)\right]
\wedge d\zeta\\
+\int_{U^{\epsilon}}f(\zeta)\cdot
\det\left[\frac{\bar\zeta}{B(\zeta,z)}\
\bar\partial\left(\frac{\bar\zeta}{B(\zeta,z)}\right)\
\bar\partial\left(\frac{\bar\zeta}{B(\zeta,z)}\right)\right]
\wedge d\zeta\Bigg].
\end{multline}
\end{lemma}
%****************************************************************
\begin{proof}
We use an alternative definition
\begin{equation}\label{omegaDefinition}
\omega^{\prime}(\eta)=\sum_{k=0}^2(-1)^{k}\eta_k\bigwedge_{j\neq k}d\eta_j
=\frac{1}{2}\det\left[\begin{tabular}{ccc}
$\eta_0$&$d\eta_0$&$d\eta_0$
\vspace{0.05in}\\
$\eta_1$&$d\eta_1$&$d\eta_1$\vspace{0.05in}\\
$\eta_2$&$d\eta_2$&$d\eta_2$
\end{tabular}\right],
\end{equation}
through determinant, which is motivated by the definition used by F. Sommer in \cite{So}
in his proof of the Bergman-Weil formula, and then described in the textbook
by B. A. Fuks \cite{Fk}. Since then this formula has been used in many articles
starting in \cite{He1, Po}.\\
\indent
Using expression ${\dis \eta(\zeta,z,\lambda)=(1-\lambda)\frac{\bar z}{B^*(\zeta,z)}
+\lambda\frac{\bar\zeta}{B(\zeta,z)} }$ in \eqref{omegaDefinition}
we obtain
\begin{multline*}
\det\bigg[\begin{tabular}{ccc}
$\eta(\zeta,z,\lambda)$&$d_{\zeta,\lambda}\eta(\zeta,z,\lambda)$
&$d_{\zeta,\lambda}\eta(\zeta,z,\lambda)$
\end{tabular}\bigg]\wedge d\zeta\\
=\Bigg((1-\lambda)\det\left[\frac{\bar z}{B^*(\zeta,z)}\
\frac{\bar\zeta d\lambda}{B(\zeta,z)}\ \frac{\lambda d\bar\zeta}{B(\zeta,z)}\right]
+(1-\lambda)\det\left[\frac{\bar z}{B^*(\zeta,z)}\
\frac{\lambda d\bar\zeta}{B(\zeta,z)}\ \frac{\bar\zeta d\lambda}{B(\zeta,z)}\right]\\
+\lambda\det\left[\frac{\bar\zeta}{B(\zeta,z)}\
\frac{-{\bar z}d\lambda}{B^*(\zeta,z)}\ \frac{\lambda d\bar\zeta}{B(\zeta,z)}\right]
+\lambda\det\left[\frac{\bar\zeta}{B(\zeta,z)}\
\frac{\lambda d\bar\zeta}{B(\zeta,z)}\ \frac{-{\bar z}d\lambda}{B^*(\zeta,z)}\right]
\Bigg)\wedge d\zeta\\
=\Bigg(2\lambda(1-\lambda)d\lambda\wedge
\det\left[\frac{\bar z}{B^*(\zeta,z)}\
\frac{\bar\zeta}{B(\zeta,z)}\ \frac{d\bar\zeta}{B(\zeta,z)}\right]
-2\lambda^2d\lambda\wedge\det\left[\frac{\bar\zeta}{B(\zeta,z)}\
\frac{\bar z}{B^*(\zeta,z)}\ \frac{d\bar\zeta}{B(\zeta,z)}\right]\Bigg)\wedge d\zeta\\
=2\lambda d\lambda\wedge\det\left[\frac{\bar z}{B^*(\zeta,z)}\
\frac{\bar\zeta}{B(\zeta,z)}\ \frac{d\bar\zeta}{B(\zeta,z)}\right]\wedge d\zeta,
\end{multline*}
and
\begin{equation*}
\omega_0^{\prime}(\eta)\wedge d\zeta=\frac{1}{2}\cdot 2\lambda d\lambda\wedge
\det\left[\frac{\bar z}{B^*(\zeta,z)}\
\frac{\bar\zeta}{B(\zeta,z)}\ \frac{d\bar\zeta}{B(\zeta,z)}\right]\wedge d\zeta.
\end{equation*}
\indent
Then we have
\begin{multline}\label{IntegratedForm}
\int_0^1\omega_0^{\prime}(\eta(\zeta,z,\lambda))\wedge d\zeta
=\left(\int_0^1\lambda d\lambda\right)
\det\left[\frac{\bar z}{B^*(\zeta,z)}\
\frac{\bar\zeta}{B(\zeta,z)}\ \frac{d\bar\zeta}{B(\zeta,z)}\right]\wedge d\zeta\\
=\frac{1}{2}\det\left[\frac{\bar z}{B^*(\zeta,z)}\
\frac{\bar\zeta}{B(\zeta,z)}\ \frac{d\bar\zeta}{B(\zeta,z)}\right]\wedge d\zeta.
\end{multline}
\indent
Using equality \eqref{IntegratedForm} we obtain formula \eqref{LFormula} for
$L^{\epsilon}\left[f\right]$:
\begin{multline*}
L^{\epsilon}\left[f\right](z)
=\frac{2}{(2\pi i)^{3}}\Bigg[\int_{bU^{\epsilon}\times[0,1]}
f(\zeta)\cdot\omega^{\prime}_0\left((1-\lambda)\frac{\bar z}
{B^*(\zeta,z)}+\lambda\frac{\bar\zeta}{B(\zeta,z)}\right)\wedge d\zeta\\
+\int_{U^{\epsilon}}f(\zeta)\cdot
\omega^{\prime}_0\left(\frac{\bar\zeta}{B(\zeta,z)}\right)\wedge d\zeta\Bigg]\\
=\frac{1}{(2\pi i)^{3}}\Bigg[\int_{bU^{\epsilon}}
f(\zeta)\cdot\det\left[\frac{\bar z}{B^*(\zeta,z)}\
\frac{\bar\zeta}{B(\zeta,z)}\ \bar\partial\left(\frac{\bar\zeta}{B(\zeta,z)}\right)\right]
\wedge d\zeta\\
+\int_{U^{\epsilon}}f(\zeta)\cdot
\det\left[\frac{\bar\zeta}{B(\zeta,z)}\
\bar\partial\left(\frac{\bar\zeta}{B(\zeta,z)}\right)\
\bar\partial\left(\frac{\bar\zeta}{B(\zeta,z)}\right)\right]
\wedge d\zeta\Bigg].
\end{multline*}
\end{proof}
%******* End Lemma Determinant  ***************************

\indent
In the following proposition we construct a Cauchy-Weil-Leray type formula on
$\epsilon$-neighborhoods of a curve in $\C\P^2$.

%****************************************************************
%******* Proposition Formula  ****************************************
\begin{proposition}\label{EpsilonFormula} Let
$$V=\left\{z\in \C\P^2:\ P(z)=0\right\}$$
be a curve in $\C\P^2$, let $U^{\epsilon}$ be as in \eqref{Uepsilon}, and
let $f$ be a holomorphic function of homogeneity $\ell$ on $U^{\epsilon}$.\\
\indent
Then for an arbitrary $z\in U^{\epsilon}$ the following equality holds
\begin{equation}\label{KEpsilonFormula}
f(z)=K^{\epsilon}\left[ f \right](z),
\end{equation}
where
\begin{multline}\label{KOperator}
K^{\epsilon}\left[f\right](z)=\frac{1}{(2\pi i)^{3}}\Bigg[\int_{\Gamma^{\epsilon}_2}
f(\zeta)\cdot\det\left[\frac{\bar z}{B^*(\zeta,z)}\
\frac{\bar\zeta}{B(\zeta,z)}\ \bar\partial\left(\frac{\bar\zeta}{B(\zeta,z)}\right)\right]
\wedge d\zeta\\
+\int_{\Gamma^{\epsilon}_1}
f(\zeta)\cdot\det\left[\frac{Q(\zeta,z)}{P(\zeta)-P(z)}\
\frac{\bar\zeta}{B(\zeta,z)}\ \bar\partial\left(\frac{\bar\zeta}{B(\zeta,z)}\right)\right]
\wedge d\zeta\\
-\int_{\Gamma^{\epsilon}_{12}}
f(\zeta)\cdot\det\left[\frac{Q(\zeta,z)}{P(\zeta)-P(z)}\
\frac{\bar z}{B^*(\zeta,z)}\ \frac{\bar\zeta}{B(\zeta,z)}\right]
\wedge d\zeta\\
+\int_{U^{\epsilon}}f(\zeta)\cdot
\det\left[\frac{\bar\zeta}{B(\zeta,z)}\
\bar\partial\left(\frac{\bar\zeta}{B(\zeta,z)}\right)\
\bar\partial\left(\frac{\bar\zeta}{B(\zeta,z)}\right)\right]
\wedge d\zeta\Bigg],
\end{multline}
$$\Gamma_1^{\epsilon}=\left\{\zeta\in \*S^{5}(1): |P(\zeta)|=\epsilon,
\varrho(\zeta)< 0\right\},\hspace{0.1in}
\Gamma_2^{\epsilon}=\left\{\zeta\in \*S^{5}(1): \varrho(\zeta)=0,
|P(\zeta)|<\epsilon\right\},$$
$$\Gamma_{12}^{\epsilon}=\left\{\zeta\in \*S^{5}(1): \varrho(\zeta)=0,
|P(\zeta)|=\epsilon\right\},$$
with coefficients $\left\{Q^i(\zeta,z)\right\}_{i=0}^2$ satisfying conditions
\eqref{HomogeneityConditions} from Lemma~\ref{WeilCoefficients}.\\
\indent
The function defined by \eqref{KOperator} on $U^{\epsilon}$
admits the descent onto a neighborhood of $V$ in $\C\P^2$.
\end{proposition}
%********************************************************************
\begin{proof}
We begin the proof of equality \eqref{KEpsilonFormula} by using equality \eqref{LFormula}
from Lemma~\ref{Determinant} with operator $K^{\epsilon}=L^{\epsilon}$:
\begin{multline}\label{PreliminaryKEpsilon}
K^{\epsilon}\left[f\right](z)
=\frac{1}{(2\pi i)^{3}}\left(
\sum_{j=1}^2\int_{\Gamma^{\epsilon}_j}
f(\zeta)\cdot\det\left[\frac{\bar z}{B^*(\zeta,z)}\
\frac{\bar\zeta}{B(\zeta,z)}\ \bar\partial\left(\frac{\bar\zeta}{B(\zeta,z)}\right)\right]
\wedge d\zeta\right.\\
+\left.\int_{U^{\epsilon}}f(\zeta)\cdot
\det\left[\frac{\bar\zeta}{B(\zeta,z)}\
\bar\partial\left(\frac{\bar\zeta}{B(\zeta,z)}\right)\
\bar\partial\left(\frac{\bar\zeta}{B(\zeta,z)}\right)\right]
\wedge d\zeta\right),
\end{multline}
and then transform the right-hand side of \eqref{PreliminaryKEpsilon} into the
right-hand side of \eqref{KOperator}.\\
\indent
To transform the integral over $\Gamma_1^{\epsilon}$ in the right-hand side of \eqref{PreliminaryKEpsilon}
we use equality
\begin{equation}\label{GeneralDeterminant}
\det\left[\begin{tabular}{cccc}
$1$&$1$&$1$&$0$
\vspace{0.05in}\\
$\eta_1(\zeta,z)$&$\eta_2(\zeta,z)$
&${\dis \frac{\bar\zeta}{B(\zeta,z)} }$
&${\dis \bar\partial\left(\frac{\bar\zeta}{B(\zeta,z)}\right) }$
\end{tabular}\right]=0 
\end{equation}
for ${\dis \eta_1(\zeta,z)=\frac{Q(\zeta,z)}{P(\zeta)-P(z)} }$,
${\dis \eta_2(\zeta,z)=\frac{\bar z}{B^*(\zeta,z)} }$, to obtain equality
\begin{multline*}
\det\left[\frac{\bar z}{B^*(\zeta,z)}\
\frac{\bar\zeta}{B(\zeta,z)}\ \frac{d\bar\zeta}{B(\zeta,z)}\right]\wedge d\zeta\\
=\det\left[\frac{Q(\zeta,z)}{P(\zeta)-P(z)}\
\frac{\bar\zeta}{B(\zeta,z)}\ \frac{d\bar\zeta}{B(\zeta,z)}\right]\wedge d\zeta
-\det\left[\frac{Q(\zeta,z)}{P(\zeta)-P(z)}\ \frac{\bar z}{B^*(\zeta,z)}\
\bar\partial\left(\frac{\bar\zeta}{B(\zeta,z)}\right)\right]\wedge d\zeta,
\end{multline*}
and, using the fact that $f$ is holomorphic in $U^{\epsilon}$, equality
\begin{multline}\label{FirstQIntegral}
\int_{\Gamma^{\epsilon}_1}
f(\zeta)\cdot\det\left[\frac{\bar z}{B^*(\zeta,z)}\
\frac{\bar\zeta}{B(\zeta,z)}\ \bar\partial\left(\frac{\bar\zeta}{B(\zeta,z)}\right)\right]
\wedge d\zeta\\
=\int_{\Gamma^{\epsilon}_1}
f(\zeta)\cdot\det\left[\frac{Q(\zeta,z)}{P(\zeta)-P(z)}\
\frac{\bar\zeta}{B(\zeta,z)}\ \bar\partial\left(\frac{\bar\zeta}{B(\zeta,z)}\right)\right]
\wedge d\zeta\\
-\int_{\Gamma^{\epsilon}_{12}}
f(\zeta)\cdot\det\left[\frac{Q(\zeta,z)}{P(\zeta)-P(z)}\
\frac{\bar z}{B^*(\zeta,z)}\ \frac{\bar\zeta}{B(\zeta,z)}\right]
\wedge d\zeta.
\end{multline}
\indent
Then using equality \eqref{FirstQIntegral} in \eqref{PreliminaryKEpsilon} we obtain
equality \eqref{KOperator}.
\end{proof}%******* End Proposition EpsilonFormula  ****************

\section{Residual integrals on $V$.}
\label{ResidualIntegrals}

\indent
In this section we analyze the behavior of the integrals
in the right-hand side of \eqref{KOperator} as $\epsilon\to 0$. In all estimates below
we use the same notation \lq\lq$C$\rq\rq\ for different constants.

\indent
For the limit of the first integral in the right-hand side of \eqref{KOperator}
we have the following lemma.\\

%**************************************************************
%******* Lemma Zero-on-Gamma2  ****************************************
\begin{lemma}\label{Zero-on-Gamma2}
Let $z\in U^{\epsilon_0}\subset\*S^{5}(1)$ for some $\epsilon_0>0$,
and let $f$ be a holomorphic function in $U^{\epsilon_0}$. Then
\begin{equation}\label{Gamma2-Equality}
\lim_{\epsilon\to 0}\int_{\Gamma^{\epsilon}_2}
f(\zeta)\cdot\det\left[\frac{\bar z}{B^*(\zeta,z)}\
\frac{\bar\zeta}{B(\zeta,z)}\ \bar\partial\left(\frac{\bar\zeta}{B(\zeta,z)}\right)\right]
\wedge d\zeta=0.
\end{equation}
\end{lemma}
%********************************************************************
\begin{proof}  Using the uniform boundedness of the kernel on $\Gamma^{\epsilon}_2$
for $z\in V$ we obtain the following estimate
\begin{equation*}
\left|\int_{\Gamma^{\epsilon}_2}
f(\zeta)\cdot\det\left[\frac{\bar z}{B^*(\zeta,z)}\ \frac{\bar\zeta}{B(\zeta,z)}\
\bar\partial\left(\frac{\bar\zeta}{B(\zeta,z)}\right)\right]
\wedge d\zeta\right|\leq C\cdot \text{Volume}\left\{\Gamma^{\epsilon}_2\right\}
\leq C\cdot\epsilon\to 0
\end{equation*}
as $\epsilon\to 0$.
\end{proof}
%******* End Lemma Zero-on-Gamma2  **********************************

\indent
For the fourth integral in the right-hand side of \eqref{KOperator} we have the following lemma.

%**************************************************************
%******* Lemma Zero-on-U  ****************************************
\begin{lemma}\label{Zero-on-U}
Let $z\in U^{\epsilon_0}\subset\*S^{5}(1)$ for some $\epsilon_0>0$,
and let $f$ be a holomorphic function in $U^{\epsilon_0}$. Then
\begin{equation}\label{Zero-U-Equality}
\lim_{\epsilon\to 0}\int_{U^{\epsilon}}f(\zeta)\cdot
\det\left[\frac{\bar\zeta}{B(\zeta,z)}\
\bar\partial\left(\frac{\bar\zeta}{B(\zeta,z)}\right)\
\bar\partial\left(\frac{\bar\zeta}{B(\zeta,z)}\right)\right]
\wedge d\zeta=0.
\end{equation}
\end{lemma}
%********************************************************************
\begin{proof}
To prove equality \eqref{Zero-U-Equality} we represent the integral in the left-hand side as
\begin{multline}\label{SumRepresentation}
\int_{U^{\epsilon}}f(\zeta)\cdot
\det\left[\frac{\bar\zeta}{B(\zeta,z)}\
\bar\partial\left(\frac{\bar\zeta}{B(\zeta,z)}\right)\
\bar\partial\left(\frac{\bar\zeta}{B(\zeta,z)}\right)\right]
\wedge d\zeta\\
=\int_{U^{\epsilon}\cap\left\{|B(\zeta,z)|>\epsilon\right\}}f(\zeta)\cdot
\det\left[\frac{\bar\zeta}{B(\zeta,z)}\
\bar\partial\left(\frac{\bar\zeta}{B(\zeta,z)}\right)\
\bar\partial\left(\frac{\bar\zeta}{B(\zeta,z)}\right)\right]
\wedge d\zeta\\
+\int_{U^{\epsilon}\cap\left\{|B(\zeta,z)|<\epsilon\right\}}
(f(\zeta)-f(z))\cdot
\det\left[\frac{\bar\zeta}{B(\zeta,z)}\
\bar\partial\left(\frac{\bar\zeta}{B(\zeta,z)}\right)\
\bar\partial\left(\frac{\bar\zeta}{B(\zeta,z)}\right)\right]
\wedge d\zeta\\
+f(z)\int_{U^{\epsilon}\cap\left\{|B(\zeta,z)|<\epsilon\right\}}
\det\left[\frac{\bar\zeta}{B(\zeta,z)}\
\bar\partial\left(\frac{\bar\zeta}{B(\zeta,z)}\right)\
\bar\partial\left(\frac{\bar\zeta}{B(\zeta,z)}\right)\right]
\wedge d\zeta.
\end{multline}
\indent
Then we obtain the following estimate for the first integral in the right-hand side
of \eqref{SumRepresentation}
\begin{multline}\label{AbsoluteEstimateOne}
\left|\int_{U^{\epsilon}\cap\left\{|B(\zeta,z)|>\epsilon\right\}}f(\zeta)\cdot
\det\left[\frac{\bar\zeta}{B(\zeta,z)}\
\bar\partial\left(\frac{\bar\zeta}{B(\zeta,z)}\right)\
\bar\partial\left(\frac{\bar\zeta}{B(\zeta,z)}\right)\right]
\wedge d\zeta\right|\\
\leq C\cdot\int_0^Adt\int_0^{\epsilon}du\int_0^{\epsilon}dv
\int_0^A\frac{rdr}{(\epsilon+t+u^2+v^2+r^2)^3}
\leq C\cdot\int_0^{\epsilon}\frac{\tau d\tau}{\epsilon+\tau^2}\\
\leq C\cdot\int_0^{\epsilon^2}\frac{ds}{\epsilon+s}
\leq C\cdot\epsilon,
\end{multline}
where we denoted $\tau=\sqrt{u^2+v^2}$, $s=\tau^2$.

\indent
For the second integral in the right-hand side of \eqref{SumRepresentation} we have
\begin{multline}\label{AbsoluteEstimateTwo}
\left|\int_{U^{\epsilon}\cap\left\{|B(\zeta,z)|<\epsilon\right\}}(f(\zeta)-f(z))\cdot
\det\left[\frac{\bar\zeta}{B(\zeta,z)}\
\bar\partial\left(\frac{\bar\zeta}{B(\zeta,z)}\right)\
\bar\partial\left(\frac{\bar\zeta}{B(\zeta,z)}\right)\right]
\wedge d\zeta\right|\\
\leq C\cdot\int_0^Adt\int_0^Adu\int_0^Adv
\int_0^{\sqrt{\epsilon}}\frac{r^2dr}{(t+u^2+v^2+r^2)^3}
\leq C\cdot\int_0^{\sqrt{\epsilon}}dr
\leq C\cdot\sqrt{\epsilon}.
\end{multline}

\indent
Using estimates \eqref{AbsoluteEstimateOne}, \eqref{AbsoluteEstimateTwo} and
equality
\begin{equation*}
\lim_{\delta\to 0}\left|\int_{|B(\zeta,z)|<\delta}
\det\left[\frac{\bar\zeta}{B(\zeta,z)}\
\bar\partial\left(\frac{\bar\zeta}{B(\zeta,z)}\right)\
\bar\partial\left(\frac{\bar\zeta}{B(\zeta,z)}\right)\right]
\wedge d\zeta\right|=0
\end{equation*}
we obtain equality \eqref{Zero-U-Equality}.
\end{proof}%******* Lemma Zero-on-U  *****************

\indent
Using now the barrier function \eqref{BoundaryBarrier} we estimate the behavior of the third
integral in the right-hand side of \eqref{KOperator} in the following proposition.

%*****************************************************************
%******* Proposition BoundaryIntegral  **********************************
\begin{proposition}\label{BoundaryIntegral} 
Let $f$ be a holomorphic function of homogeneity $\ell<0$ on $U^{\epsilon}$
for some $\epsilon>0$. Then for $z\in V$ the following equality holds
\begin{multline}\label{FirstIntegral}
\lim_{\epsilon\to 0}\int_{\Gamma^{\epsilon}_{12}}
f(\zeta)\cdot\det\left[\frac{Q(\zeta,z)}{P(\zeta)-P(z)}\
\frac{\bar z}{B^*(\zeta,z)}\ \frac{\bar\zeta}{B(\zeta,z)}\right]d\zeta\\
=-\lim_{\epsilon\to 0}\int_{\Gamma^{\epsilon}_{12}}f(\zeta)\cdot
\det\left[\frac{R(z)}{F(z,\zeta)}\ \frac{Q(\zeta,z)}{P(\zeta)-P(z)}\
\frac{{\bar\zeta}}{B(\zeta,z)}\right]d\zeta,
\end{multline}
where $F(z,\zeta)=\sum_{i=0}^2R_i(z)(\zeta_i-z_i)$
is the barrier function from \eqref{BoundaryBarrier}.
\end{proposition}
%********************************************************************
\begin{proof}
Using equality
\begin{equation*}
\det\left[\begin{tabular}{cccc}
$1$&$1$&$1$&$1$
\vspace{0.05in}\\
${\dis \frac{Q(\zeta,z)}{P(\zeta)-P(z)} }$&${\dis \frac{\bar z}{B^*(\zeta,z)} }$
&${\dis \frac{\bar\zeta}{B(\zeta,z)} }$
&${\dis \frac{R(z)}{F(z,\zeta)} }$
\end{tabular}\right]=0,
\end{equation*}
its corollary
\begin{multline*}
\det\left[\frac{{\bar z}}{B^*(\zeta,z)}\
\frac{\bar\zeta}{B(\zeta,z)}\ \frac{R(z)}{F(z,\zeta)}\right]
-\det\left[\frac{Q(\zeta,z)}{P(\zeta)-P(z)}\ \frac{{\bar\zeta}}{B(\zeta,z)}\
\frac{R(z)}{F(z,\zeta)}\right]\\
+\det\left[\frac{Q(\zeta,z)}{P(\zeta)-P(z)}\ \frac{{\bar z}}{B^*(\zeta,z)}
\frac{R(z)}{F(z,\zeta)}\right]
-\det\left[\frac{Q(\zeta,z)}{P(\zeta)-P(z)}\ \frac{{\bar z}}{B^*(\zeta,z)}\
\frac{{\bar\zeta}}{B(\zeta,z)}\right]=0,
\end{multline*}
and estimate
\begin{equation*}
\left|\int_{\Gamma^{\epsilon}_{12}}f(\zeta)\cdot
\det\left[\frac{{\bar z}}{B^*(\zeta,z)}\ \frac{\bar\zeta}{B(\zeta,z)}\
\frac{R(z)}{F(z,\zeta)}\right]\right|
\leq C\cdot\epsilon \to 0
\end{equation*}
for $z\in V$ as $\epsilon\to 0$, we obtain equality
\begin{multline}\label{StillTwoIntegrals}
\lim_{\epsilon\to 0}\int_{\Gamma^{\epsilon}_{12}}
f(\zeta)\cdot\det\left[\frac{Q(\zeta,z)}{P(\zeta)-P(z)}\
\frac{\bar z}{B^*(\zeta,z)}\ \frac{\bar\zeta}{B(\zeta,z)}\right]d\zeta\\
=-\lim_{\epsilon\to 0}\int_{\Gamma^{\epsilon}_{12}}f(\zeta)\cdot
\det\left[\frac{Q(\zeta,z)}{P(\zeta)-P(z)}\ \frac{{\bar\zeta}}{B(\zeta,z)}\
\frac{R(z)}{F(z,\zeta)}\right]d\zeta\\
+\lim_{\epsilon\to 0}\int_{\Gamma^{\epsilon}_{12}}f(\zeta)\cdot
\det\left[\frac{Q(\zeta,z)}{P(\zeta)-P(z)}\ \frac{{\bar z}}{B^*(\zeta,z)}\
\frac{R(z)}{F(z,\zeta)}\right]d\zeta.
\end{multline}
\indent
To simplify the right-hand side of equality \eqref{StillTwoIntegrals} we need the following lemma.
%*****************************************************************
%******* Lemma  ZeroAtInfinity ***************************************
\begin{lemma}\label{ZeroAtInfinity} 
If $f$ is a holomorphic function of homogeneity $\ell<0$ on $U^{\epsilon}$
for some $\epsilon>0$, then for $z\in V$
\begin{equation}\label{HolomorphicZero}
\int_{\Gamma^{\epsilon}_{12}}f(\zeta)\cdot
\det\left[\frac{Q(\zeta,z)}{P(\zeta)-P(z)}\ \frac{{\bar z}}{B^*(\zeta,z)}\
\frac{R(z)}{F(z,\zeta)}\right]d\zeta=0.
\end{equation}
\end{lemma}
%********************************************************************
\begin{proof}
To evaluate for $z\in V$ the integral
\begin{multline*}
\int_{\Gamma^{\epsilon}_{12}}f(\zeta)\cdot
\det\left[\frac{Q(\zeta,z)}{P(\zeta)-P(z)}\ \frac{{\bar z}}{B^*(\zeta,z)}\
\frac{R(z)}{F(z,\zeta)}\right]d\zeta\\
=\int_{\Gamma^{\epsilon}_{12}}f(\zeta)\cdot
\det\left[\frac{Q(\zeta,z)}{P(\zeta)}\ \frac{{\bar z}}{B^*(\zeta,z)}\
\frac{R(z)}{F(z,\zeta)}\right]d\zeta
\end{multline*}
we consider the restriction of the function $f$ to $V$, and its lift to
$\sigma^{-1}(V)\subset\sigma^{-1}\left(\C\P^2\right)=\C^3\setminus\{0\}$,
which we also denote by $f$.\\
\indent
We use the following estimates for $|\zeta|\to\infty$
\begin{equation}\label{HomogeneityEstimates}
\left|\frac{\bar z}{B^*(\zeta,z)}\right|,\
\left|\frac{Q(\zeta,z)}{P(\zeta)}\right|,\
\left|\frac{R(z)}{F(z,\zeta)}\right| \leq \frac{C}{|\zeta|},
\end{equation}
and notice that for all values of $a>0$ the sets
$$\Gamma^{\epsilon}_{12}( a)
=\left\{\zeta\in \*S^{5}(a): |P(\zeta)|=\epsilon\cdot a^{\deg P},\
\varrho(\zeta)=0\right\}$$
are real analytic subvarieties of $\*S^{5}(a)$ of real dimension $3$ satisfying
\begin{equation}\label{VolumeEstimate}
c\cdot a^{3}\cdot\text{Volume}\left(\Gamma^{\epsilon}_{12}\right)
<\text{Volume}\left(\Gamma^{\epsilon}_{12}(a)\right)
<C\cdot a^{3}\cdot\text{Volume}\left(\Gamma^{\epsilon}_{12}\right).
\end{equation}
\indent
We use the Stokes' formula for the form
\begin{equation*}
f(\zeta)\cdot\det\left[\frac{\bar z}{B^*(\zeta,z)}\ \frac{Q(\zeta,z)}{P(\zeta)}\
\frac{R(z)}{F(z,\zeta)}\right]d\zeta,
\end{equation*}
which is holomorphic with respect to $\zeta$, on the variety
$$\left\{\zeta\in \C^3:
\left|P(\zeta)\right|=\epsilon\cdot|\zeta|^{\deg P},\
\varrho(\zeta)=0,\ 1<|\zeta|<a\right\}$$
with the boundary
\begin{equation*}
\Gamma^{\epsilon}_{12}\cup \Gamma^{\epsilon}_{12}(a).
\end{equation*}
\indent
Then for $f$ satisfying $f(\lambda\zeta)=\lambda^{\ell}f(\zeta)$ with $\ell<0$
and $z\in V$, using estimates \eqref{HomogeneityEstimates} and
\eqref{VolumeEstimate} we obtain the following estimate
\begin{multline*}
\left|\int_{\Gamma^{\epsilon}_{12}}f(\zeta)\cdot
\det\left[\frac{Q(\zeta,z)}{P(\zeta)}\ \frac{{\bar z}}{B^*(\zeta,z)}\
\frac{R(z)}{F(z,\zeta)}\right]d\zeta\right|\\
=\left|\int_{\Gamma^{\epsilon}_{12}(a)}f(\zeta)\cdot
\det\left[\frac{Q(\zeta,z)}{P(\zeta)}\ \frac{{\bar z}}{B^*(\zeta,z)}\
\frac{R(z)}{F(z,\zeta)}\right]d\zeta\right|
\leq C\cdot\frac{\text{Volume}\left(\Gamma^{\epsilon}_{12}(a)\right)}{a^4}\to 0
\end{multline*}
as $a\to \infty$.
\end{proof}%******* Lemma  ZeroAtInfinity *********************************
\indent
Applying Lemma~\ref{ZeroAtInfinity} to the second integral in
the right-hand side of equality \eqref{StillTwoIntegrals}, we obtain equality \eqref{FirstIntegral}.
\end{proof}%******* Proposition BoundaryIntegral  ***********************

\indent
To estimate the second integral of the right-hand side of \eqref{KOperator}
we use the following proposition.
%*****************************************************************
%******* Proposition Gamma1Integral  **********************************
\begin{proposition}\label{Gamma1Integral} 
For a holomorphic function $f$ on $U^{\epsilon}$ and $z\in V$ the following equality holds
\begin{multline}\label{SecondIntegral}
\lim_{\epsilon\to 0}\int_{\Gamma^{\epsilon}_1}
f(\zeta)\cdot\det\left[\frac{Q(\zeta,z)}{P(\zeta)-P(z)}\
\frac{\bar\zeta}{B(\zeta,z)}\ \bar\partial\left(\frac{\bar\zeta}{B(\zeta,z)}\right)\right]
\wedge d\zeta\\
=-\lim_{\epsilon\to 0}\lim_{\delta\to 0}\int_{\Gamma^{\epsilon,\delta}_{12}(z)}
f(\zeta)\cdot\det\left[\frac{R(z)}{F(z,\zeta)}\
\frac{Q(\zeta,z)}{P(\zeta)-P(z)}\ \frac{\bar\zeta}{B(\zeta,z)}\right]d\zeta\\
-\lim_{\epsilon\to 0}\int_{\Gamma^{\epsilon}_{12}}
f(\zeta)\cdot\det\left[\frac{R(z)}{F(z,\zeta)}\
\frac{Q(\zeta,z)}{P(\zeta)-P(z)}\ \frac{\bar\zeta}{B(\zeta,z)}\right]d\zeta
\end{multline}
where
\begin{equation*}
\Gamma^{\epsilon,\delta}_{12}(z)
=\left\{\zeta\in \*S^{5}(1): |P(\zeta)|=\epsilon,\ |F(z,\zeta)|=\delta\right\}.
\end{equation*}
\end{proposition}
%********************************************************************
\begin{proof} Using equality \eqref{GeneralDeterminant}
with ${\dis \eta_1(\zeta,z)=\frac{R(z)}{F(z,\zeta)} }$,
${\dis \eta_2(\zeta,z)=\frac{Q(\zeta,z)}{P(\zeta)-P(z)} }$,
we obtain equality
\begin{multline*}
\det\left[\frac{Q(\zeta,z)}{P(\zeta)-P(z)}\
\frac{\bar\zeta}{B(\zeta,z)}\ \frac{d\bar\zeta}{B(\zeta,z)}\right]\wedge d\zeta\\
=\det\left[\frac{R(z)}{F(z,\zeta)}\
\frac{\bar\zeta}{B(\zeta,z)}\ \frac{d\bar\zeta}{B(\zeta,z)}\right]\wedge d\zeta
-\det\left[\frac{R(z)}{F(z,\zeta)}\ \frac{Q(\zeta,z)}{P(\zeta)-P(z)}\
\bar\partial\left(\frac{\bar\zeta}{B(\zeta,z)}\right)\right]\wedge d\zeta\\
=\det\left[\frac{R(z)}{F(z,\zeta)}\
\frac{\bar\zeta}{B(\zeta,z)}\ \frac{d\bar\zeta}{B(\zeta,z)}\right]\wedge d\zeta
-\bar\partial\left(\det\left[\frac{R(z)}{F(z,\zeta)}\ \frac{Q(\zeta,z)}{P(\zeta)-P(z)}\
\frac{\bar\zeta}{B(\zeta,z)}\right]\right)\wedge d\zeta.
\end{multline*}
\indent
Then, using the Stokes' formula on
\begin{equation*}
\Gamma^{\epsilon,\delta}_1(z)=\left\{\zeta\in \*S^{5}(1): |P(\zeta)|=\epsilon,\
|F(z,\zeta)|>\delta\right\}
\end{equation*}
for the second term of the right-hand side of equality above,
and holomorphic dependence of $f(\zeta)$ on $\zeta$ we obtain  the following equality
\begin{multline}\label{SecondQIntegral}
\int_{\Gamma^{\epsilon}_1}
f(\zeta)\cdot\det\left[\frac{Q(\zeta,z)}{P(\zeta)-P(z)}\
\frac{\bar\zeta}{B(\zeta,z)}\ \bar\partial\left(\frac{\bar\zeta}{B(\zeta,z)}\right)\right]
\wedge d\zeta\\
=\lim_{\delta\to 0}\int_{\Gamma^{\epsilon,\delta}_1(z)}
f(\zeta)\cdot\det\left[\frac{Q(\zeta,z)}{P(\zeta)-P(z)}\
\frac{\bar\zeta}{B(\zeta,z)}\ \bar\partial\left(\frac{\bar\zeta}{B(\zeta,z)}\right)\right]
\wedge d\zeta\\
=\lim_{\delta\to 0}\int_{\Gamma^{\epsilon,\delta}_1(z)}
f(\zeta)\cdot\det\left[\frac{R(z)}{F(z,\zeta)}\
\frac{\bar\zeta}{B(\zeta,z)}\ \bar\partial\left(\frac{\bar\zeta}{B(\zeta,z)}\right)\right]
\wedge d\zeta\\
-\lim_{\delta\to 0}\int_{\Gamma^{\epsilon,\delta}_{12}(z)}
f(\zeta)\cdot\det\left[\frac{R(z)}{F(z,\zeta)}\
\frac{Q(\zeta,z)}{P(\zeta)-P(z)}\ \frac{\bar\zeta}{B(\zeta,z)}\right]d\zeta\\
-\int_{\Gamma^{\epsilon}_{12}}
f(\zeta)\cdot\det\left[\frac{R(z)}{F(z,\zeta)}\
\frac{Q(\zeta,z)}{P(\zeta)-P(z)}\ \frac{\bar\zeta}{B(\zeta,z)}\right]d\zeta.
\end{multline}

\indent
In the following lemma we estimate the behavior of the first integral in the right-hand
side of \eqref{SecondQIntegral} as $\epsilon\to 0$.
%*****************************************************************
%******* Lemma ZeroSurface  ****************************************
\begin{lemma}\label{ZeroSurface}
Let $f$ be a holomorphic function on $U^{\epsilon}$ for some $\epsilon>0$. Then
for $z\in V$ we have
\begin{equation}\label{ZeroSurfaceEquality}
\lim_{\epsilon\to 0}\lim_{\delta\to 0}\int_{\Gamma^{\epsilon,\delta}_1}
f(\zeta)\cdot\det\left[\frac{R(z)}{F(z,\zeta)}\
\frac{\bar\zeta}{B(\zeta,z)}\
\bar\partial\left(\frac{\bar\zeta}{B(\zeta,z)}\right)\right]
\wedge d\zeta=0.
\end{equation}
\end{lemma}
%****************************************************************
\begin{proof}
Denoting
\begin{equation*}
\begin{aligned}
&U^{\epsilon,\delta}_2(z)=\left\{\zeta\in \*S^{5}(1): |F(z,\zeta)|>\delta,\
|P(\zeta)|<\epsilon\right\},\vspace{0.1in}\\
&\Gamma^{\epsilon,\delta}_2(z)=\left\{\zeta\in \*S^{5}(1): |F(z,\zeta)|=\delta,\
|P(\zeta)|<\epsilon\right\},
\end{aligned}
\end{equation*}
we apply the Stokes' formula to the form
\begin{equation*}
f(\zeta)\cdot\det\left[\frac{R(z)}{F(z,\zeta)}\
\frac{\bar\zeta}{B(\zeta,z)}\
\frac{d\bar\zeta}{B(\zeta,z)}\right]\wedge d\zeta
\end{equation*}
on the domain $U^{\epsilon,\delta}_2(z)$.
Then we obtain equality
\begin{multline}\label{SurfaceEquality}
\int_{\Gamma_1^{\epsilon,\delta}}f(\zeta)
\cdot\det\left[\frac{R(z)}{F(z,\zeta)}\
\frac{\bar\zeta}{B(\zeta,z)}\
\frac{d\bar\zeta}{B(\zeta,z)}\right]\wedge d\zeta\\
+\int_{\Gamma_2^{\epsilon}}f(\zeta)
\cdot\det\left[\frac{R(z)}{F(z,\zeta)}\
\frac{\bar\zeta}{B(\zeta,z)}\
\frac{d\bar\zeta}{B(\zeta,z)}\right]\wedge d\zeta\\
+\int_{\Gamma_2^{\epsilon,\delta}(z)}f(\zeta)
\cdot\det\left[\frac{R(z)}{F(z,\zeta)}\
\frac{\bar\zeta}{B(\zeta,z)}\
\frac{d\bar\zeta}{B(\zeta,z)}\right]\wedge d\zeta\\
=\int_{U^{\epsilon,\delta}_2(z)}
f(\zeta)\cdot\det\left[\frac{R(z)}{F(z,\zeta)}\
\bar\partial\left(\frac{\bar\zeta}{B(\zeta,z)}\right)\
\bar\partial\left(\frac{\bar\zeta}{B(\zeta,z)}\right)\right]\wedge d\zeta.
\end{multline}

\indent
For the third integral in the left-hand side of \eqref{SurfaceEquality} we have
\begin{multline*}
\left|\int_{\Gamma_2^{\epsilon,\delta}(z)}
f(\zeta)\cdot\det\left[\frac{R(z)}{F(z,\zeta)}\
\frac{\bar\zeta}{B(\zeta,z)}\
\frac{d\bar\zeta}{B(\zeta,z)}\right]\wedge d\zeta\right|
\leq C\cdot\left|\int_0^Adt\int_0^{\epsilon}d\rho
\int_{|w|=\delta}\frac{\rho dw}
{\delta\cdot(\epsilon+t+\rho^2)^2}\right|\\
\leq C\cdot\int_0^{\epsilon}\frac{\rho d\rho}
{(\epsilon+\rho^2)}
\leq C\cdot\int_0^{\epsilon^2}\frac{du}{\epsilon+u}\to 0
\end{multline*}
as $\epsilon\to 0$.

\indent
For the second integral in the left-hand side of \eqref{SurfaceEquality}
using uniform boundedness of the kernel on $\Gamma^{\epsilon}_2$
for $z\in V$ we obtain the following estimate
\begin{equation*}
\left|\int_{\Gamma^{\epsilon}_2}
f(\zeta)\cdot\det\left[\frac{R(z)}{F(z,\zeta)}\
\frac{\bar\zeta}{B(\zeta,z)}\
\frac{d\bar\zeta}{B(\zeta,z)}\right]\wedge d\zeta\right|
\leq C\cdot \text{Volume}\left\{\Gamma^{\epsilon}_2\right\}
\leq C\cdot\epsilon\to 0
\end{equation*}
as $\epsilon\to 0$.\\
\indent
For the integral in the right-hand side of equality \eqref{SurfaceEquality}
we use the equality
\begin{multline*}
\lim_{\epsilon\to 0}\lim_{\delta\to 0}\left|\int_{U^{\epsilon,\delta}_2(z)}
f(\zeta)\cdot\det\left[\frac{R(z)}{F(z,\zeta)}\
\bar\partial\left(\frac{\bar\zeta}{B(\zeta,z)}\right)\
\bar\partial\left(\frac{\bar\zeta}{B(\zeta,z)}\right)\right]
\wedge d\zeta\right|\\
=\lim_{\epsilon\to 0}\lim_{\delta\to 0}\left|\int_{U^{\epsilon,\delta}_2(z)}
f(\zeta)\cdot\det\left[\frac{\bar\zeta}{B(\zeta,z)}\
\bar\partial\left(\frac{\bar\zeta}{B(\zeta,z)}\right)\
\bar\partial\left(\frac{\bar\zeta}{B(\zeta,z)}\right)\right]
\wedge d\zeta\right|=0,
\end{multline*}
where in the second equality we used Lemma~\ref{Zero-on-U}.\\
\indent
From the above estimates follows equality \eqref{ZeroSurfaceEquality}.
\end{proof}%******* Lemma ZeroSurface  **********************************
Using equality \eqref{ZeroSurfaceEquality} in \eqref{SecondQIntegral} we obtain equality
\eqref{SecondIntegral} from Proposition~\ref{Gamma1Integral}.
\end{proof}%******* Proposition Gamma1Integral  *************************

\section{Proof of Theorem~\ref{Main}.}

\indent
Combining Propositions~\ref{BoundaryIntegral} and \ref{Gamma1Integral} we obtain from
formula \eqref{KOperator} the following formula for $z\in V$:
\begin{multline}\label{SecondKOperator}
\lim_{\epsilon\to 0}K^{\epsilon}\left[f\right](z)
=\frac{1}{(2\pi i)^{3}}\lim_{\epsilon\to 0}\Bigg[\int_{\Gamma^{\epsilon}_1}
f(\zeta)\cdot\det\left[\frac{Q(\zeta,z)}{P(\zeta)}\
\frac{\bar\zeta}{B(\zeta,z)}\ \bar\partial\left(\frac{\bar\zeta}{B(\zeta,z)}\right)\right]
\wedge d\zeta\\
-\int_{\Gamma^{\epsilon}_{12}}
f(\zeta)\cdot\det\left[\frac{Q(\zeta,z)}{P(\zeta)}\
\frac{\bar z}{B^*(\zeta,z)}\ \frac{\bar\zeta}{B(\zeta,z)}\right]d\zeta\Bigg]\\
=\frac{1}{(2\pi i)^{3}}\lim_{\epsilon\to 0}\Bigg[
-\lim_{\delta\to 0}\int_{\Gamma^{\epsilon,\delta}_{12}(z)}
f(\zeta)\cdot\det\left[\frac{R(z)}{F(z,\zeta)}\
\frac{Q(\zeta,z)}{P(\zeta)}\ \frac{\bar\zeta}{B(\zeta,z)}\right]d\zeta\\
+2\int_{\Gamma^{\epsilon}_{12}}
f(\zeta)\cdot\det\left[\frac{Q(\zeta,z)}{P(\zeta)}\
\frac{R(z)}{F(z,\zeta)}\ \frac{\bar\zeta}{B(\zeta,z)}\right]d\zeta\Bigg].
\end{multline}

\indent
We further transform formula \eqref{SecondKOperator} into an integral formula with integral
taken over the boundary and the kernel holomorphically depending on $z$ by explicitly
computing limits in the first integral of its right-hand side in the following lemma.
%*****************************************************************
%******* Lemma PointResidues  ****************************************
\begin{lemma}\label{PointResidues}
Let $z\in V$ be fixed, let $w\in {\cal V}_z$, and let $f$ be a holomorphic function
on $U^{\epsilon}$. Then the following equality holds
\begin{equation}\label{ResidueValue}
\frac{1}{(2\pi i)^{3}}\lim_{\epsilon\to 0}
\lim_{\delta\to 0}\int_{\Gamma^{\epsilon,\delta}_{12}(w)}
f(\zeta)\cdot\det\left[\frac{R(w)}{F(w,\zeta)}\
\frac{Q(\zeta,w)}{P(\zeta)}\ \frac{\bar\zeta}{B(\zeta,w)}\right]d\zeta
=\sum_{j=0}^pf(w^{(j)}),
\end{equation}
where $w_0=w,w_1,\dots,w_p$ are the points from \eqref{RConditions} (ii) satisfying equality $F(w,w^{(j)})=0$.
\end{lemma}
%****************************************************************
\begin{proof}
We represent the integral in the left-hand side of equality \eqref{ResidueValue} as
\begin{multline}\label{SumResidues}
\frac{1}{(2\pi i)^{3}}\lim_{\epsilon\to 0}
\lim_{\delta\to 0}\int_{\Gamma^{\epsilon,\delta}_{12}(w)}
f(\zeta)\cdot\det\left[\frac{R(w)}{F(w,\zeta)}\
\frac{Q(\zeta,w)}{P(\zeta)}\ \frac{\bar\zeta}{B(\zeta,w)}\right]d\zeta\\
=\sum_{j=0}^p\frac{1}{(2\pi i)^{3}}\lim_{\epsilon\to 0}
\lim_{\delta\to 0}\int_{\Gamma^{\epsilon,\delta}_{12}(w^{(j)})}
f(\zeta)\cdot\det\left[\frac{R(w)}{F(w,\zeta)}\
\frac{Q(\zeta,w)}{P(\zeta)}\ \frac{\bar\zeta}{B(\zeta,w)}\right]d\zeta,
\end{multline}
where
$\Gamma^{\epsilon,\delta}_{12}(w^{(j)})=\Gamma^{\epsilon,\delta}_{12}(w)\cap U_j$,
and $U_j$ is a small enough neighborhood of the point $w^{(j)}$.\\
\indent
For each $w^{(j)}\ (j=0,\dots,p)$ we have
\begin{multline}\label{OneDIntegral}
\frac{1}{(2\pi i)^{3}}\lim_{\epsilon\to 0}
\lim_{\delta\to 0}\int_{\Gamma^{\epsilon,\delta}_{12}(w^{(j)})}
f(\zeta)\cdot\det\left[\frac{R(w)}{F(w,\zeta)}\
\frac{Q(\zeta,w)}{P(\zeta)}\ \frac{\bar\zeta}{B(\zeta,w)}\right]d\zeta\\
=\frac{1}{(2\pi i)^{3}}\lim_{\epsilon\to 0}
\lim_{\delta\to 0}\int_{\Gamma^{\epsilon,\delta}_{12}(w^{(j)})}
f(\zeta)\frac{d_{\zeta}F(w^{(j)},\zeta)}{F(w^{(j)},\zeta)}
\wedge\frac{d_{\zeta}P(\zeta)}{P(\zeta)}
\wedge\frac{d_{\zeta}B(\zeta,w^{(j)})}{B(\zeta,w^{(j)})}\\
=\frac{1}{(2\pi i)}\int_{\C^1(w^{(j)})\cap\left\{|\zeta|=1\right\}}
f(\zeta)\cdot\frac{d_{\zeta}B(\zeta,w^{(j)})}{B(\zeta,w^{(j)})},
\end{multline}
where in the last integral we integrate over the unit circle in
$$\C^1(w^{j)})=\left\{\zeta=\mu\cdot w^{(j)}\right\}.$$
\indent
To evaluate the last integral we substitute variables
$\zeta_1=\lambda \zeta_0, \zeta_2=\lambda \zeta_0$ and obtain equality
\begin{equation*}
B(\zeta,z)={\bar\zeta}_0(\zeta_0-w_0)
+{\bar\lambda}_1\lambda_1{\bar\zeta}_0(\zeta_0-w_0)
+{\bar\lambda}_2\lambda_2{\bar\zeta}_0(\zeta_0-w_0)
={\bar\zeta_0}(\zeta_0-w_0)(1+{\bar\lambda}_1\lambda_1+{\bar\lambda}_2\lambda_2).
\end{equation*}
Using this equality in equality \eqref{OneDIntegral} we obtain
\begin{equation*}
\frac{1}{(2\pi i)}\int_{\C^1(w^{(j)})\cap\left\{|\zeta|=1\right\}}
f(\zeta)\cdot\frac{d_{\zeta}B(\zeta,w^{(j)})}{B(\zeta,w^{(j)})}
=\frac{1}{(2\pi i)}\int_{\C^1(w^{(j)})\cap\left\{|\zeta|=1\right\}}
f(\zeta)\cdot\frac{d\zeta_0}{\zeta_0-w^{(j)}_0}=f(w^{(j)}).
\end{equation*}
\end{proof}

\indent
Using equality \eqref{ResidueValue} in formula \eqref{SecondKOperator} we obtain
for $w\in {\cal V}_z$ equality
\begin{equation*}
\lim_{\epsilon\to 0}K^{\epsilon}\left[f\right](w)
=-\sum_{j=0}^pf(w^{(j)})
+\frac{2}{(2\pi i)^{3}}\lim_{\epsilon\to 0}\int_{\Gamma^{\epsilon}_{12}}
f(\zeta)\cdot\det\left[\frac{Q(\zeta,w)}{P(\zeta)}\
\frac{R(w)}{F(w,\zeta)}\ \frac{\bar\zeta}{B(\zeta,w)}\right]d\zeta,
\end{equation*}
and combining it with \eqref{KEpsilonFormula} obtain equality
\begin{equation}\label{MultiFormula}
2f(w)+\sum_{j=1}^pf(w^{(j)})
=\frac{2}{(2\pi i)^{3}}\lim_{\epsilon\to 0}\int_{\Gamma^{\epsilon}_{12}}
f(\zeta)\cdot\det\left[\frac{Q(\zeta,w)}{P(\zeta)}\
\frac{R(w)}{F(w,\zeta)}\ \frac{\bar\zeta}{B(\zeta,w)}\right]d\zeta.
\end{equation}

\indent
Using equality \eqref{MultiFormula}
we construct the following system of linear equations with Vandermonde matrix
for the values of the function $f$ at the points
$\left\{w,w^{(1)},\dots,w^{(p)}\right\}$:
\begin{equation}\label{MultiSystem}
\left[\begin{tabular}{cccc}
1&1&$\cdots$&1\vspace{0.05in}\\
${\dis \frac{w_1}{w_0} }$&${\dis \frac{w^{(1)}_1}{w_0} }$&$\cdots$&
${\dis \frac{w^{(p)}_1}{w^{(p)}_0} }$\vspace{0.05in}\\
$\vdots$&$\vdots$&&$\vdots$\vspace{0.05in}\\
${\dis \left(\frac{w_1}{w_0}\right)^p }$&
${\dis \left(\frac{w^{(1)}_1}{w_0}\right)^p }$&$\cdots$&
${\dis \left(\frac{w^{(p)}_1}{w^{(p)}_0}\right)^p }$
\end{tabular}\right]
\left[\begin{tabular}{c}
$2f(w)$\vspace{0.22in}\\
$f(w^{(1)})$\vspace{0.22in}\\
$\vdots$\vspace{0.22in}\\
$f(w^{(p)})$
\end{tabular}\right]
=\left[\begin{tabular}{c}
$G_0(w)$\vspace{0.22in}\\
$G_1(w)$\vspace{0.22in}\\
$\vdots$\vspace{0.22in}\\
$G_p(w)$
\end{tabular}\right],
\end{equation}
where
\begin{equation*}
G_k(w)=\frac{2}{(2\pi i)^{3}}\lim_{\epsilon\to 0}\int_{\Gamma^{\epsilon}_{12}}
f(\zeta)\cdot\left(\frac{\zeta_1}{\zeta_0}\right)^k\cdot\det\left[\frac{Q(\zeta,w)}{P(\zeta)}\
\frac{R(w)}{F(w,\zeta)}\ \frac{\bar\zeta}{B(\zeta,w)}\right]d\zeta.
\end{equation*}
\indent
By summing individual systems of the form \eqref{MultiSystem}
\begin{equation}\label{JMultiSystem}
\left[\begin{tabular}{cccc}
1&1&$\cdots$&1\vspace{0.05in}\\
${\dis \frac{w_1}{w_0} }$&${\dis \frac{w^{(1)}_1}{w_0} }$&$\cdots$&
${\dis \frac{w^{(p)}_1}{w^{(p)}_0} }$\vspace{0.05in}\\
$\vdots$&$\vdots$&&$\vdots$\vspace{0.05in}\\
${\dis \left(\frac{w_1}{w_0}\right)^p }$&
${\dis \left(\frac{w^{(1)}_1}{w_0}\right)^p }$&$\cdots$&
${\dis \left(\frac{w^{(p)}_1}{w^{(p)}_0}\right)^p }$
\end{tabular}\right]
\left[\begin{tabular}{c}
$f(w)$\vspace{0.1in}\\
$\vdots$\vspace{0.1in}\\
$2f(w^{(j)})$\vspace{0.1in}\\
$\vdots$\vspace{0.1in}\\
$f(w^{(p)})$
\end{tabular}\right]
=\left[\begin{tabular}{c}
$G_0(w^{(j)})$\vspace{0.1in}\\
$\vdots$\vspace{0.1in}\\
$G_j(w^{(j)})$\vspace{0.1in}\\
$\vdots$\vspace{0.1in}\\
$G_p(w^{(j)})$
\end{tabular}\right],
\end{equation}
we construct a \lq\lq symmetrized\rq\rq\ version of \eqref{MultiSystem}
\begin{equation}\label{SymmetricSystem}
A(w)\cdot f(w)= \frac{p+1}{p+2}\cdot G(w,w^{(1)},\dots,w^{(p)}),
\end{equation}
where $A$ is the Vandermonde matrix of system \eqref{MultiSystem}, and
\begin{equation}
G_k(w,w^{(1)},\dots,w^{(p)})=\sum_{j=0}^pG_k(w^{(j)}).
\end{equation}
Then using the Cramer's rule for system \eqref{SymmetricSystem}
we obtain equality \eqref{fValues}.

\end{document}